\documentclass{amsart}
\usepackage{amssymb}
\usepackage{stmaryrd}
\usepackage{mathrsfs}
\usepackage{fullpage}

\newtheorem{theorem}{Theorem}[section]
\newtheorem{lemma}[theorem]{Lemma}

\theoremstyle{definition}

\theoremstyle{remark}

\numberwithin{equation}{section}

\newcommand{\FF}{{\mathbb{F}}}

\newcommand{\bC}{{\mathbf{C}}}
\newcommand{\bZ}{{\mathbf{Z}}}
\newcommand{\bF}{{\mathbf{F}}}
\newcommand{\bN}{{\mathbf{N}}}

\newcommand{\Aut}{{\operatorname{Aut}}}

\newcommand{\PSL}{{\operatorname{PSL}}}

\newcommand{\GL}{{\operatorname{GL}}}

\newcommand{\Ker}{{\operatorname{Ker}}}
\newcommand{\Out}{{\operatorname{Out}}}
\newcommand{\End}{{\operatorname{End}}}

\newcommand{\GF}{\mbox{GF}}

\let\nor=\triangleleft

\begin{document}

\title{On nilpotent and solvable quotients of primitive groups}

\author{Thomas Michael Keller}
\author{Yong Yang}
\address{Department of Mathematics, Texas State University, 601 University Drive, San Marcos, TX 78666, USA.}

\makeatletter
\email{keller@txstate.edu, yang@txstate.edu}
\makeatother

\date{}


\Large
\begin{abstract}
It is shown that if $G$ is a primitive permutation group on a set of size $n$, then any nilpotent quotient of $G$ has order at most $n^{\beta}$ and any solvable quotient of $G$ has order at most $n^{\alpha+1}$ where $\beta=\log 32/ \log 9$ and $\alpha=(3 \log (48)+\log (24))/ (3 \cdot \log (9))$. This was motivated by a result of Aschbacher and Guralnick ~\cite{AschGural}.
\end{abstract}

\maketitle
\section{Introduction} \label{sec:introduction8}
In ~\cite{AschGural}, Aschbacher and Guralnick studied the maximum abelian quotient of permutation groups and linear groups. They showed the following results.

Theorem $1^*$. Let $G$ be a primitive permutation group on a set of finite order $n$. Then $|G:G'| \leq  n$.

Theorem $2^*$. Let $G$ be a permutation group on a finite set of order $n$. Then $|G:G'| \leq 2^{n-1}$.

Theorem $3^*$. Let $V$ be a finite dimensional vector space over a finite field of characteristic $p$ and $G$ a subgroup of $\GL(V)$. If $O_p(G)=1$, then $|G:G'| < |V|$.

It is natural to ask what one can say about the maximum nilpotent quotient or maximum solvable quotient in the same context. Let $G$ be a finite group. We denote $G^*$ to be the normal subgroup of $G$ such that $G/G^*$ is the maximum nilpotent quotient. We denote $G^S$ to be the normal subgroup of $G$ such that $G/G^S$ is the maximum solvable quotient. Let $\beta=\log 32/ \log 9$, $\alpha=(3 \log (48)+\log (24))/ (3 \cdot \log (9))$ and $\lambda=(24)^{1/3}$. In this note, we show the following results.

Theorem 1. Let $G$ be a primitive permutation group on a set of finite order $n$. Then $|G: G^*| \leq  n^{\beta}/2$ and $|G: G^S| \leq  n^{\alpha+1}/\lambda$.

Theorem 2. Let $G$ be a permutation group on a finite set of order $n$. Then $|G:G^*| \leq 2^{n-1}$ and $|G:G^S| \leq \lambda^{n-1}$.

Theorem 3. Let $V$ be a finite dimensional vector space over a finite field of characteristic $p$ and $G$ a subgroup of $\GL(V)$. If $O_p(G)=1$, then $|G:G^*| \leq |V|^{\beta}/2$ and $|G:G^S| \leq |V|^{\alpha}/\lambda$.


Theorems 2 and 3 are needed to prove Theorem 1. Theorem 3 depends on the classification of simple groups (one needs to know the relative sizes of outer automorphism groups, degrees of permutation representations, and dimensions of modules). 


The results of Aschbacher and Guralnick were motivated in answering a question of Tamagawa, which was prompted by ~\cite{IsaacsProblem}. A minor modification of the argument in ~\cite{IsaacsProblem} shows that $|G: G^*(H \cap H^g)| \leq |G: H|$. This was first observed by D. Cantor in the following form: If $\alpha$ and $\beta$ are conjugate over a field $K$, then any abelian extension $L/K$ with $L \subset K(\alpha, \beta)$ satisfies $|L: K| \leq |K(\alpha): K|$. The field theoretic version of Theorem 1 is:

Corollary 4. Let $K$ be a field with $M/K$ a minimal extension. If $L/K$ is an nilpotent extension with $L$ contained in the normal closure of $M$, then $|L: K| \leq |M:K|^{\beta}/2$. If $L/K$ is a solvable extension with $L$ contained in the normal closure of $M$, then $|L: K| \leq |M:K|^{\alpha+1}/\lambda$.

In order to prove these results, we need the following Lemmas.

\begin{lemma} \label{normalinduction}
Let $\pi=G \rightarrow A$ be a surjective group homomorphism and $K = \Ker \pi$. Then $|G:G^*|=|A: A^*||K \cap G^*| \leq |A:A^*||K:K^*|$ and $|G:G^S|=|A: A^S||K \cap G^S| \leq |A:A^S||K:K^S|$.
\end{lemma}

\begin{theorem}\label{maxsubgroup}
Let $G$ be a group which is not cyclic of prime order. Then $G$ has a proper subgroup $H$ such that $|H| \geq |G|^{1/2}$.
\end{theorem}
\begin{proof}
This is the main theorem of ~\cite{ALev}.
\end{proof}

\begin{lemma} \label{nilpotentaffine}
  Let $G$ be a finite group and $V$ a finite irreducible faithful $G$-module. Then $(GV)^*=G^* V$.
\end{lemma}
\begin{proof}
Clearly $(GV)/(G^{*} V)$ is nilpotent, so $(GV)^*$ is a subgroup of $G^{*} V$.

Also clearly $G^*$ is a subgroup of $(GV)^*$, so it remains to show that $V$ is a subgroup of $(GV)^*$.

Let $W$ be the intersection of $V$ and $(GV)^*$. We want to show that $W=V$ and, working towards a contradiction, assume $W<V$.

Since $G$ acts on $W$ and $V$ is irreducible, we conclude that $W=1$. Therefore $(GV)^*$ and $V$ are normal subgroups of $GV$ which have trivial intersection. Hence $[(GV)^*,V]=1$, and as $V$ is a faithful $G$-module, this implies that $(GV)^* \leq V$. So as $W=1$, we altogether get $(GV)^*=1$, so that $GV$ is nilpotent. As $V$ is a faithful $G$-module, this forces $GV$ to be a $p$-group, where $p$ is the characteristic of $V$. But a $p$-group does not have a faithful irreducible module in characteristic $p$, so $GV=V$. Then $W=V>1$, a contradiction.
\end{proof}

\section{Main Theorems} \label{sec:maintheorem}

\begin{theorem} \label{pernilpotent}
Let $G$ be a group of permutations on a set $S$ of order $n$. Then $|G: G^*| \leq 2^{n-1}$.
\end{theorem}
\begin{proof}
The proof is by induction on $|G|$. So assume $G$ is a minimal counterexample.

Step 1. $G$ is nilpotent. If not, choose $M$ maximal in $G$ with $G = M G^*$. Then $|G:G^*| =|M: M \cap G^*| \leq |M:M^*| \leq 2^{n-1}$  by induction.

Step 2. $G$ is transitive on $S$. If $G$ acts on a proper subset $Y$ of $S$ with $0 \neq m = |Y| < n$, then apply Lemma ~\ref{normalinduction} with $\pi$ the representation of $G$ on $Y$ to obtain $|G: G^*| \leq 2^{m-1} 2^{(n-m-1)} < 2^{n-1}$. Similarly, if $G$ is a $p$-group, we see that $G$ is transitive.

Step 3. $G$ is a $p$-group for some prime $p$. If not then $G= P \times Q$ where $P$ is a $p$-group and $Q$ is a $p'$-group. Let $H = G_x$ be the stabilizer of $x$ in $S$. Then $H = M \times N$ with $M = H \cap P$ and $N = H \cap Q$. Note that $P$ acts faithfully on the cosets of $M$ as does $Q$ on the cosets of $N$. So $|G: G^*| = |P: P^*||Q: Q^*| \leq 2^{a-1} 2^{b-1} < 2^{n-1}$, where $a = |P:M|$ and $b = |Q:N|$.

Step 4. Since $n$ is a positive integer, we write $n$ to the base $p$:

$n=n_0+n_1 p+n_2 p^2 \dots + n_k p^k$ where $0 \leq n_i <p$ for each $i$.

Then the Sylow $p$-subgroups of $S_n$ have order $p^{\nu(n)}$ where $\nu(n)=n_1 + n_2 (p^2-1)/(p-1) + \dots + n_k (p^k-1)/(p-1)$. If $n_i \geq 1$, then $p^{n_i(p^i-1)/(p-1)} \leq 2^{n_i (p^i -1)} \leq 2^{n_i p^i -1}$ for all $p$. This implies that $p^{\nu(n)} \leq 2^{n-1}$ for all $p$.
\end{proof}

\begin{theorem} \label{primitivepersolvable}
Let $G$ be a primitive permutation group on a set $S$ of order $n$. Then $|G: G^S| \leq {\lambda}^{n-1}$.
\end{theorem}
\begin{proof}
By ~\cite[Corollary 1.4]{Maroti}, we know that $G$ either contains $A_n$ or is one of the groups in the exceptional list. If $G$ contains $A_n$, then $|G:G^S| \leq 2 \leq \lambda^{n-1}$. If $G$ is one of the groups in the exceptional list, it is easy to check using GAP~\cite{GAP} that it satisfies the requirement.
\end{proof}

\begin{theorem} \label{persolvable}
Let $G$ be a group of permutations on a set $S$ of order $n$. Then $|G: G^S| \leq {\lambda}^{n-1}$.
\end{theorem}
\begin{proof}
The proof is by induction on $|G|$. So assume $G$ is a minimal counterexample. Thus $G$ is primitive on $V$.
Proof. If not, there is a nontrivial decomposition $S = \bigcup_i S_i$ with $G$ permuting $X = \{S_1, \dots , S_r \}$ and $\bN_G(S_1)$ acting primitively on $S_1$ where $|S_1|=m$ and $n=m r$. Let $\pi$ be the permutation representation of $G$ on $X$ and $K = \ker \pi$. Set $K_0=K$ and $K_{i+1} =\{g \in K_i \ |\ g$ acts trivially on $S_{i+1} \}$. By Lemma ~\ref{normalinduction}, $|G: G^S| = s_0 s_1 \dots s_r$, where $s_0 =|G:G^S K|$ and $s_i = |K_{i-1} : (K_{i-1} \cap G^S) K_i| \leq |K_{i-1}: K_{i-1}^S K_i|$ for $i \geq 1$. By induction, $s_0 \leq \lambda^{r-1}$, $s_i \leq |K_{i-1}:K_{i-1}^S K_i| \leq \lambda^{m-1}$, thus $|G: G^S| \leq \lambda^{n-1}$.
\end{proof}

\begin{lemma} \label{boundofouter}
Let $L$ be a finite nonabelian simple group and $M$ a proper subgroup of $L$ with $|L: M| = k$.
\begin{enumerate}
\item If $L \not \cong A_6$, then $2|\Out L|<k$.
\item $2|\Out L|<k^{\beta}$, $\lambda |\Out L|<k^{\alpha} \leq |L|^{\alpha/2}$ and $2|\Out L| < \sqrt {|L|}$.
\item If $L \leq \PSL_d(q)$, then $2 |\Out L| <q^d/(q-1)$.
\end{enumerate}
\end{lemma}
\begin{proof}
This follows from ~\cite[Lemma 2.7]{AschGural} and Theorem ~\ref{maxsubgroup}.
\end{proof}

\begin{theorem} \label{linearbound}
Let $\FF$ be a finite field of characteristic $p$, $V$ is a finite dimensional vector space over $\FF$, and $G \leq \GL(V)$. Assume $O_p(G) = 1$. Then $|G: G^*| \leq |V|^{\beta}/2$ and $|G: G^S| \leq |V|^{\alpha}/\lambda$.
\end{theorem}
\begin{proof}
Let $(G,V)$ be a minimal counterexample to

(4.1). Clearly $G$ is nonabelian.

(4.2). $G$ acts irreducibly on $V$.
Proof. If not, choose $0 = V_0< V_1 < \dots < V_d = V$ with $V_i$ $G$-invariant and $W_i = V_i/V_{i-1}$ $G$-irreducible. Since $O_p(G)=1$, $G$ acts faithfully on $W = W_1 \oplus \dots \oplus W_d$. So we can assume $V$ is a semisimple module. Let $U$ be an irreducible submodule where $\pi$ is the representation of $G$ on $U$. Thus if $K = \ker \pi$, $|G: G^*K| \leq |U|^{\beta}/2$ and $|G: G^S K| \leq |U|^{\alpha}/\lambda$. Since $O_p(K) = 1$, $K$ acts faithfully on $V/U$. So by minimality, $|K: K^*| \leq |V/U|^{\beta}/2$ and $|K: K^S| \leq |V/U|^{\alpha}/\lambda$, whence $|G: G^*| \leq |V|^{\beta}/2$ and $|G: G^S| \leq |V|^{\alpha}/\lambda$. 

(4.3). $G$ is primitive on $V$.
Proof. If not, there is a nontrivial decomposition $V = \oplus_i V_i$ with $G$ permuting $X = \{V_1, \dots , V_r \}$ and $\bN_G(V_1)$ acting primitively on $V_1$. Let $\pi$ be the permutation representation of $G$ on $X$ and $K = \ker \pi$. Set $K_0=K$ and $K_{i+1} =\{g \in K_i \ |\ g$ acts trivially on $V_{i+1} \}$. By Lemma ~\ref{normalinduction}, $|G: G^*| = k_0 k_1 \dots k_r$, where $k_0 =|G:G^*K|$ and $k_i = |K_{i-1} : (K_{i-1} \cap G^*) K_i| \leq |K_{i-1}: K_{i-1}^* K_i|$ for $i \geq 1$. By Lemma ~\ref{normalinduction}, $|G: G^S| = s_0 s_1 \dots s_r$, where $s_0 =|G:G^S K|$ and $k_i = |K_{i-1} : (K_{i-1} \cap G^S) K_i| \leq |K_{i-1}: K_{i-1}^S K_i|$ for $i \geq 1$. By Theorem ~\ref{pernilpotent}, $k_0 \leq 2^{r-1}$. By Theorem ~\ref{persolvable}, $s_0 \leq \lambda^{r-1}$.

Since $\bN_G(V_i)$ acts primitively on $V_i$ and $K \nor \bN_G(V_i)$, $K$ acts homogeneously on each $V_i$. If $K$ is not irreducible on $V_1$ (and so not on each $V_i$), then $k_i \leq |K_{i-1}:K_{i-1}^*| \leq |W|^{\beta}/2 \leq |V_i|^{\beta}/2$ and $s_i \leq |K_{i-1}:K_{i-1}^S| \leq |W|^{\alpha}/\lambda \leq |V_i|^{\alpha}/\lambda$, where $W$ is a $K$-irreducible subspace of $V_1$. Thus $|G: G^*| \leq 2^{r-1}(|V_i|^{\beta}/2)^r = |V|^{\beta}/2$ and $|G: G^S| \leq \lambda^{r-1}(|V_i|^{\alpha}/\lambda)^r = |V|^{\alpha}/\lambda$.

So assume $K$ acts irreducibly on $V_i$, then $k_i \leq |K_{i-1}:K_{i-1}^*K_i| \leq |V_i|^{\beta}/2$ and $s_i \leq |K_{i-1}:K_{i-1}^S K_i| \leq |V_i|^{\alpha}/\lambda$ by induction, and $|G: G^*| \leq |V|^{\beta}/2$, $|G: G^S| \leq |V|^{\alpha}/\lambda$ as above. 

(4.4). If $D \nor G$, then $D$ acts homogeneously. In particular, if $D$ is abelian, then $D$ is cyclic.
Proof. Apply (4.3).

(4.5). If $D$ is a noncyclic normal subgroup of $G$, then $D$ acts irreducibly on $V$ and $\bC_G(D)$ is cyclic.
Proof. By (4.4), $D$ acts homogeneously. So assume $V = V_1 \oplus \cdots \oplus V_m$, $V_i \cong V_j$ as irreducible $D$-modules. Let $E = \End_{\FF D} V_1$ with $q = |E|$ and $m > 1$. Set $U =V_1$. Let $\alpha: D \rightarrow \GL (U)$ be the representation of $D$ on $U$. Let $\Gamma$ be the normalizer of $D \alpha$ in $\GL(U)$ and $S$ the centralizer of $D \alpha$. Since $gV_1 \cong V_1$ as $\FF D$-modules for each $g \in G$, we can define $\pi : G \rightarrow \Gamma/S$ by $g \pi = S x$ where $\alpha(g^{-1} h g) = x^{-1} \alpha(h) x$. Note $K = \ker \pi = \bC_G(D)$. Let $M$ be the inverse image of $G \pi$ in $\Gamma \leq \GL(U)$. By minimality, $|G: G^*K| \leq |M: M^*| \leq |U|^{\beta}/2$ and $|G: G^S K| \leq |M: M^S| \leq |U|^{\alpha}/\lambda$. Since $K = \bC_G(D)$, there exists a faithful representation $\beta: K \rightarrow \GL_m(E)$. By minimality, $|K:K^*| \leq q^{\beta m}/2$ and $|K:K^s|\leq q^{\alpha m}/\lambda$. Thus $|G: G^*| \leq |U|^{\beta} q^{\beta m}/2 \leq |U|^{\beta} q^{\beta (2m-2)}/2 \leq |U|^{\beta m}/2 = |V|^{\beta}/2$ and $|G: G^S| \leq |U|^{\alpha} q^{\alpha m}/ \lambda \leq |U|^{\alpha} q^{\alpha (2m-2)}/ \lambda \leq |U|^{\alpha m}/\lambda = |V|^{\alpha}/\lambda$.

Let $T$ be a maximal normal cyclic subgroup of $G$. If $T=\bC_G(T)$, then $G \leq \Gamma(V)$ and $|G| \mid n(p^n-1)$. It is clear that $n (p^n-1) \leq p^{n \alpha}/\lambda$ for all $p,n$. It is also clear that $n (p^n-1) \leq p^{n \beta}/2$ for all $p,n$ except when $p=2, n=2,3,4,5$ or $p=3,n=2$.

Thus we may assume that $T \neq \bC_G(T)$.

Now choose $D$ minimal subject to $D \leq \bC_G(T)$, $D \nor G$, and $D$ is nonabelian. Thus by (4.4), (4.5), and the maximality of $T$ :

(4.7). $\bZ(D)$ is the largest characteristic abelian subgroup of $D$, $\bZ(D)$ is cyclic, $D$ acts irreducibly on $V$, and $T = \bC_G(D)$.

(4.8). Either

(4.8.1) $D$ is the central product of quasisimple subgroups $L_1, \dots ,L_r$ permuted transitively by $G$, or

(4.8.2) $D = P \bZ(D)$ with $|P| = p^{1+2w}$ and $P$ an extra-special $p$-group and either $\bZ(D) = \bZ(P)$ or $p = 2$ and $\bZ(D) \cong Z_4$. Moreover, $G$ is irreducible on $\tilde{D} = D/\bZ(D)$.

Proof. This follows from (4.7) and the minimal choice of $D$.

(4.9). (4.8.2) holds.
Proof. Assume (4.8.1) holds. Let $X = \{L, \dots ,L_r \}$, $\pi$ the permutation representation of $G$ on $X$ and $K = \ker \pi$. By (4.5), $D$ acts irreducibly on $V$. Let $E = \End_D V$ and set $q = |E|$. Since $V$ is absolutely irreducible over $E$, $V = V_1 \otimes_E \cdots \otimes_E V_r$ where $V_i$ is an absolutely irreducible $\hat{L}_i$-module (where $\hat{L}_i$ is the covering group of $L_i$). Let $d = \dim_E V_i$. Then $|V| = q^{d^r}$ and $\End_{L_1} V_1 = E$. By Lemma ~\ref{boundofouter}(iii), $2k < q^d/(q - 1)$, where $k = |\Out L_1|$. In particular, $(2k)^r (q - 1) < q^{d^r} = |V|$. Since $T = \bC_G(D)$ by (4.7), it follows that $|T| <q$, $|K:K^*(T \cap K)| \leq k^r$ and $|K:K^S(T \cap K)| \leq k^r$. By Theorem ~\ref{pernilpotent}, $|G: G^*K| \leq 2^{r-1}$. By Theorem ~\ref{persolvable}, $|G: G^S K| \leq \lambda^{r-1}$. Thus $|G: G^*| \leq 2^{r-1} k^r(q-1) = (2k)^r (q-1)/2 \leq |V|^{\beta}/2$ and $|G: G^S| \leq \lambda^{r-1} k^r(q-1) = (\lambda k)^r (q-1)/\lambda \leq |V|^{\alpha}/\lambda$.

(4.10). Let $q = |\End_D V|$. Then $q > |\bC_G(D)|$, $|V| \geq q^{p^w}$, and $q \equiv 1 \mod p$.

Proof. Clearly $q > |T|$. Since $F$ is a faithful $D$-module, the second inequality holds. Since $\Aut_p V \geq \bZ(D)$, $q \equiv 1 \mod p$.

(4.11). $\bC_G(\tilde{D}) = D \bC_G(D)$.

Proof. Any automorphism of $D$ which is trivial on $D$ must be inner.

Now $G$ acts irreducibly on $\tilde{D}$. Since $D \leq G^* \bZ(D)$ by Lemma ~\ref{nilpotentaffine}, it follows by the minimality of $(G,V)$ that $|G: G^*\bC_G(\tilde{D})| \leq |\tilde{D}|^{\beta}/2$. Thus $|G: G^*| \leq (|\tilde{D}|^{\beta}/2) \cdot |\bC_G(D): \bC_G(D) \cap G^*| \leq p^{2 \beta w}/2 \cdot (q-1)< q^{\beta p^w}/2 \leq |V|^{\beta }/2$. Also, $|G: G^S| \leq (|\tilde{D}|^{\alpha}/\lambda) \cdot |D \bC_G(D)| \leq p^{2 \alpha w}/\lambda \cdot p^{2w}(q-1)< q^{\alpha p^w}/\lambda \leq |V|^{\alpha}/\lambda$ unless $p=2$, $q=3$, $w=1$. For this special case, one may check the result by direct calculation.
\end{proof}

\begin{lemma} \label{primitivefirst}
Let $G$ be a primitive permutation group of degree $n$ on a set $S$. Then either
\begin{enumerate}
\item $|G: G^*| \leq n^{\beta}/2$ and $|G: G^S| \leq n^{\alpha}/\lambda$; or
\item  $G$ preserves an affine structure on $S$.
\end{enumerate}
\end{lemma}
\begin{proof}
Assume (2) does not hold. First consider the case where $\bF^*(G) = D = L_1 \times \dots \times L_r$ is the product of $r$ nonabelian simple groups $L_i$ permuted transitively by $G$. It follows by ~\cite[Theorem 1]{AschScott} that either $n = k^r$ for $k = |L_1: M|$ for some proper subgroup $M$ of $L_1$ or that $n =l^s \geq l^{r/2}$ where $l= |L_1|$ and $r \geq 2$. Let $K = \bigcap_i \bN_G(L_i)$. By Theorem ~\ref{pernilpotent}, $|G: G^*K| <2^{r-1}$. Moreover, $K/D \leq \Out L_1 \times \cdots \times \Out L_r.$ If $r = 1$, then $|G: G^*| \leq |\Out L_1| < k/2$ by Lemma ~\ref{boundofouter}(1) unless $L_1 \cong A_6$. If $L_1 \cong A_6$, then either $k = 6$ and $|G:G^*| \leq 2$ or $k > 8$ and $|G: G^*| \leq 4$. If $r> 1$, then $|G:G^*| = |G: G^*K||K: K \cap G^*| \leq 2^{r-1} |\Out L_1|^{r}$ and $|G:G^S| = |G: G^S K||K: K \cap G^S| \leq \lambda^{r-1} |\Out L_1|^{r}$. If $n = k^r$, then Lemma ~\ref{boundofouter}(2) implies $|G: G^*| \leq n^{\beta}/2$ and $|G: G^S| \leq n^{\alpha}/\lambda$. If $n =l^s$, $s \geq r/2$, then $|G:G^*| \leq (2|\Out L_1|)^{r}/2 < l^{r/2}/2 <n^{\beta}/2$ and $|G:G^S| \leq (\lambda|\Out L_1|)^{r}/\lambda < l^{\alpha r/2}/ \lambda \leq n^{\alpha}/\lambda$ by Lemma ~\ref{boundofouter} (2).


By ~\cite{AschScott}, the only case left is when $\bF^*(G) = Q = R \times D$, where $R \cong D$ and $D$ is as above. Moreover, if $H$ is a point stabilizer, then $G = HD$ and $H \cap Q = \{(x , \lambda(x))\ |\ x \in R \}$ for some isomorphism $\lambda$ of $R$ and $D$. In particular, $n = |D| = l^r$ where $l = |L_1|$. Arguing as above, we see that $|G: G^*|\leq l^{r/2}/2<n^{\beta}/2$ and $|G: G^S|\leq l^{\alpha r/2}/ \lambda < n^{\alpha}/\lambda$.
\end{proof}

\begin{theorem} \label{thm4}
Let $G$ be a primitive group of permutations on a set $S$ of order $n$. Then $|G: G^*| \leq n^{\beta}/2$ and $|G: G^S| \leq n^{\alpha+1}/\lambda$.
\end{theorem}
\begin{proof}
By Lemma ~\ref{primitivefirst}, $G$ preserves an affine space structure on $S$.
Thus $\bF^*(G) = V$ is an $m$-dimensional vector space over $\FF = \GF(p)$ for some prime $p$, with $V$ regular on $S$ and $G = HV$, where $H = G_x$ is the stabilizer of some $x \in X$ and $H$ is an irreducible subgroup of $\GL(V)$. Then $n = |V|$. We can assume $H \neq 1$.

$|G: G^*| = |H: H^*|$ by Lemma ~\ref{nilpotentaffine} and $|H: H^*| < n^{\beta}/2$ by Theorem ~\ref{linearbound}. 
$|G: G^S| \leq |H: H^S||V| \leq n^{\alpha+1}/\lambda$ by Theorem ~\ref{linearbound}.
\end{proof}

\section{Examples} \label{sec:Examples}
In this section, we shall provide some examples to show that our results are in some sense the best possible.

Remark: Theorem 3 may be viewed as the generalization of ~\cite[Theorem 3.3 and Theorem 3.5(a)]{MAWOLF}. ~\cite[Theorem 3.3]{MAWOLF} is the best possible as the following example shows us. $G \cong \GL(2,3) \wr S_4$ acts on $V=\FF_3^8$. ~\cite[Theorem 3.5(a)]{MAWOLF} is the best possible as the following example shows us. $G \cong \Gamma(3^2)$ acts on $V=\FF_3^2$.\\

\
\centerline{\bf Acknowledgements \rm}
\\
The second author would like to thank for financial support from the AMS-Simons travel grant.

\medskip


\end{document}